\documentclass[11pt]{amsart}
\usepackage[
    colorlinks=true,       
    linkcolor=magenta,          
    citecolor=magenta,        
    filecolor=magenta,      
    urlcolor=magenta,           
    pagebackref=true]{hyperref}

\usepackage{stmaryrd,mathrsfs}
\usepackage{xcolor}
\numberwithin{equation}{section}
\allowdisplaybreaks[4]
\usepackage{amsmath,amsthm,amscd,amssymb,amsfonts, amsbsy}
\usepackage{latexsym}
\usepackage{exscale}
\newtheorem{Theorem}[equation]{Theorem}
\newtheorem{Proposition}[equation]{Proposition}

\newtheorem{Remark}[equation]{Remark}

\newtheorem{Conjecture}[equation]{Conjecture}

\def\Xint#1{\mathchoice
{\XXint\displaystyle\textstyle{#1}}%
{\XXint\textstyle\scriptstyle{#1}}%
{\XXint\scriptstyle\scriptscriptstyle{#1}}%
{\XXint\scriptscriptstyle\scriptscriptstyle{#1}}%
\!\int}
\def\XXint#1#2#3{{\setbox0=\hbox{$#1{#2#3}{\int}$}
\vcenter{\hbox{$#2#3$}}\kern-.5\wd0}}

\def\dashint{\Xint-}

\def\bbR{\mathbb{R}}

\def\ch{\rm{ch}}
\def\dw{\textup{d}w}

\begin{document}

\title[Two weight inequalities for bilinear forms]{Two weight inequalities for bilinear forms}

\author{Kangwei Li}

\address{Department of Mathematics and Statistics, P.O.B. 68 (Gustaf
H\"all\-str\"omin katu 2b), FI-00014 University of Helsinki, Finland}
\email{kangwei.li@helsinki.fi}

\thanks{The author is supported by the European Union through Tuomas Hyt\"onen's ERC
Starting Grant
``Analytic-probabilistic methods for borderline singular integrals''.
He is a member of the Finnish Centre of Excellence in Analysis and
Dynamics Research.
}

\date{\today}


\begin{abstract}
Let $1\le p_0<p,q <q_0\le \infty$. Given a pair of weights $(w,\sigma)$ and a sparse family $\mathcal S$,  we study the two weight inequality for the following bi-sublinear form
\[
B(f, g)= \sum_{Q\in\mathcal S}\langle |f|^{p_0}\rangle_Q^{\frac 1{p_0}} \langle|g|^{q_0'}\rangle_Q^{\frac 1{q_0'}}\lambda_Q\le \mathcal N\|f\|_{L^{p}(w)}\|g\|_{L^{q'}(\sigma)}.
\]
When $\lambda_Q=|Q|$ and $p=q$, Bernicot, Frey and Petermichl showed that $B(f,g)$ dominates $\langle Tf, g\rangle$ for a large class of singular non-kernel operators. We give a characterization for the above inequality and then obtain the mixed $A_p$-$A_\infty$ estimates and the corresponding entropy bounds when $\lambda_Q=|Q|$ and $p=q$. We also proposed a new conjecture which implies both the one supremum conjecture and the separated bump conjecture.
\end{abstract}

\maketitle

\section{Introduction and main results}\label{Sect:1}
The weighted theory for Calder\'on-Zygmund operators has achieved several advances in the last decades. The quantitative relation between the weighted bound of the operator and the $A_p$ characteristics has attracted many authors' interest. The climax of this topic is the settle of the $A_2$ conjecture, which was due to Hyt\"onen \cite{Hytonen2012}. Lerner \cite{Lerner} also gave a simple proof for it by reducing the problem to study the so-called sparse operators. We refer the readers to \cite{Hytonen2012,Lerner} and the reference therein for an overview of this topic.

Later on, Hyt\"onen and Lacey \cite{HL2012} extended the $A_2$ theorem to the mixed $A_p$-$A_\infty$ estimate, and proposed the famous one supremum conjecture (which will be recalled below), which is in the theme of ``finding the minimal sufficient condition such that the two weight inequality holds''.
As far as we know, this conjecture is still open. Another problem in the same theme is the so-called separated bump conjecture (which will be recalled in Section \ref{sec:5}), which arised from work of Cruz-Uribe and P\'erez \cite{CP1, CP2}, and Cruz-Uribe and Reznikov and Volberg \cite{CRV}. The separated bump conjecture was just verified for the log-bumps \cite{CRV, ACM}. In general, it is still unknown, see also in \cite{Lacey} for more details.

Recently,
 Bernicot, Frey and Petermichl \cite{BFP} studied the weighted theory beyond Calder\'on-Zygmund theory. They gave the sharp weighted estimates for a large class of singular non-kernel operators by proving a domination theorem. To be precise, they showed that, if $T$ is bounded on $L^2$ and satisfies some cancellation property and Cotlar type inequality, then
for $f,g$ supported in $5Q_0$ for some cube $Q_0$,   there exists some sparse family $\mathcal S_0$ such that
\begin{equation}\label{eq:domination}
\Big|\int_{Q_0} Tf \cdot g d\mu\Big|\le C \sum_{P\in \mathcal S_0}\mu(P) \langle |f|^{p_0}\rangle_{5P}^{1/{p_0}}\langle |g|^{q_0'}\rangle_{5P}^{1/{q_0'}},
\end{equation}
where
\[
\langle h\rangle_Q=\frac 1{\mu(Q)}\int_Q h(x) d\mu,
\]
and $1\le p_0<q_0\le \infty$. Recall that we say $\mathcal S$ is sparse if  for any $Q\in \mathcal S$,
\[
\mu\Big(\bigcup_{ Q' \in \mathcal S, Q'\subsetneq Q }Q'\Big)\le \frac 12 \mu(Q).
\]
The above result is in the setting of locally compact separable metric space equipped
with a doubling Borel measure $\mu$, finite on compact sets and strictly positive on any nonempty
open set. In the following, we simply use $|Q|$ stands for $\mu(Q)$.

It is well-known that there exists  a finite collection of adjacent dyadic systems $\mathcal D_k$, $k=1,\cdots, K$, such that for any cube $Q$, there exists some $1\le k_0\le K$ such that $5Q\subset \tilde Q\in \mathcal D_{k_0}$ and $|\tilde Q|\eqsim |Q|$ (e.g. see \cite{HK}). Then RHS\eqref{eq:domination} can be dominated by
\[
\int_{\bbR^n} \sum_{k=1}^K\sum_{Q\in \mathcal D_k}  \langle |f|^{p_0}\rangle_{ Q}^{1/{p_0}}\langle |g|^{q_0'}\rangle_{Q}^{1/{q_0'}}\sum_{\substack{P\in \mathcal S_0\\ \tilde P=Q}} \mathbf 1_{P}(x) dx.
\]
Then following similar arguments as that in \cite{Hanninen}, it suffices to consider the following bi-sublinear form
\[
B(f, g)= \sum_{Q\in\mathcal S}\langle |f|^{p_0}\rangle_Q^{\frac 1{p_0}} \langle|g|^{q_0'}\rangle_Q^{\frac 1{q_0'}}\lambda_Q,
\]
where  $\mathcal S$ is a sparse family (here we generalize $|Q|$ to general sequence $\lambda_Q$). This is the main object in this paper.

Our first result concern the characterization of two weight norm inequality for the bilinear form.
\begin{Theorem}\label{thm:characterization}
Let $(w,\sigma)$ be a pair of weights and $p_0<p,q<q_0$. Suppose that $\mathcal N$ is the best constant such that the following two weight inequality holds
\begin{equation}\label{eq:norm}
B(f,g)\le \mathcal N \|f\|_{L^p(w)}\|g\|_{L^{q'}(\sigma)}.
\end{equation}
Denote
\begin{eqnarray*}
\tau_Q= \langle u\rangle_Q^{-\frac 1{p_0'}} \langle v\rangle_Q^{-\frac 1{q_0}} \frac{\lambda_Q}{|Q|},\,\, u:= w^{\frac {p_0}{p_0-p}},\,\,   v:=\sigma^{\frac{q_0'}{q_0'-q'}},
\end{eqnarray*}
and
\[
T_F(f)=\sum_{\substack{Q\in\mathcal S\\ Q\subset F}}  \tau_Q \langle f \rangle_Q \mathbf 1_Q.
\]
Then we have
\begin{equation*}
\mathcal N
  \eqsim \mathfrak{T}_s+\mathfrak{T}_s^*,
\end{equation*}
where $s\in(1,\infty]$ is determined by
\begin{equation}\label{eq:rDef}
  \frac{1}{s}=\Big(\frac{1}{q}-\frac{1}{p}\Big)_+
  :=\max\Big\{\frac{1}{q}-\frac{1}{p},0\Big\},
\end{equation}
and
\begin{equation}\label{eq:newTesting}
\begin{split}
  \mathfrak{T}_s & :=\sup_{\mathcal F}\Big \|\Big\{\frac{\|T_{F}(u)\|_{L^{q}(v)}}{u(F)^{1/p}}\Big\}_{F\in\mathcal F} \Big\|_{\ell^s}, \\
  \mathfrak{T}_s^* & :=\sup_{\mathcal{G}}\Big \|\Big\{\frac{\|T_{G}(v)\|_{L^{p'}(u)}}{v(G)^{1/{q'}}}\Big\}_{G\in\mathcal G} \Big\|_{\ell^s},
\end{split}
\end{equation}
and the supremums are taken over all subcollections $\mathcal F$ and $\mathcal{G}$ of $\mathcal{S}$ that are sparse with respect to $u$ and $v$, respectively.
\end{Theorem}
For the case of $p_0=1, q_0=\infty$, this result is already known in \cite{LSU} for $p\le q$ and \cite{tanaka2014} for $p>q$, see also in \cite{HHL} for an unified approach for both cases. In this sense, our result extends from  the special case $p_0=1, q_0=\infty$ to general cases.
Next, we are concerned with the mixed $A_p$-$A_\infty$ type estimate.
\begin{Theorem}\label{thm:sharpconstant}
Let $(w,\sigma)$ be a pair of weights, $\lambda_Q=|Q|$, $p=q$ and  $\mathcal N$ be the best constant such that \eqref{eq:norm} holds.
Then
\[
\mathcal N\lesssim [v]_{A_r}^{\frac 1{q_0'}-\frac 1{p'}}([u]_{A_\infty}^{\frac 1p}+ [v]_{A_\infty}^{\frac 1{p'}}),
\]
where
\[
r=\Big(\frac {q_0}p\Big)'(\frac p{p_0}-1)+1.
\]
\end{Theorem}
We point out Theorem \ref{thm:sharpconstant} improves the main result in \cite{BFP}. Indeed,
notice that Bernicot, Frey and Petermichl showed that, if $\sigma=w^{1-p'}$, i.e., $u=v^{1-r'}$, then
\[
\mathcal N\lesssim [v]_{A_r}^{\max\{\frac 1{q_0'}, \frac 1{p_0(r-1)}\}}.
\]
Since
\[
[u]_{A_\infty}^{\frac 1p}\le [u]_{A_{r'}}^{\frac 1p}=[v]_{A_r}^{\frac 1{p(r-1)}},\,\,\mbox{and}\,\, [v]_{A_\infty}^{\frac 1{p'}}\le [v]_{A_r}^{\frac 1{p'}},
\]
we have
\[
[v]_{A_r}^{\frac 1{q_0'}-\frac 1{p'}}([u]_{A_\infty}^{\frac 1p}+ [v]_{A_\infty}^{\frac 1{p'}})\le [v]_{A_r}^{\frac 1{p_0(r-1)}}+ [v]_{A_r}^{\frac 1{q_0'}}\le 2[v]_{A_r}^{\max\{\frac 1{q_0'}, \frac 1{p_0(r-1)}\}}.
\]

Finally we study the one supremum estimate. In \cite{HL2012}, Hyt\"onen and Lacey proposed the following one supremum conjecture, which is an extension of the $A_p$-$A_\infty$ estimate.
\begin{Conjecture}
Let $T$ be a Calder\'on-Zygmund operator, $w,\sigma$ be weights and $1<p<\infty$. Then
\[
\|T(\cdot\sigma)\|_{L^p(\sigma)\rightarrow L^p(w)}\lesssim \sup_Q \langle w\rangle_Q^{\frac 1p}\langle \sigma\rangle_Q^{\frac 1{p'}}(A_\infty (w,Q)^{\frac 1{p'}}+A_\infty (\sigma,Q)^{\frac 1p}),
\]
where for any weight $v$,
\[
 A_\infty (v,Q):= \frac{1}{v(Q)}\int_Q M(v\chi_Q)(x) dx.
\]
\end{Conjecture} Here we give some partial answer to this conjecture for the general indices $1\le p_0<q_0\le \infty$. Define
\[
A_r(v,u,Q)= \langle v\rangle_Q^{\frac 1 p-\frac 1{q_0}}\langle u\rangle_Q^{(\frac 1p-\frac 1{q_0})(r-1)},
\]
and
\begin{align*}
[v,u]_{(A_r) (\phi(A_\infty)) }&:= \sup_Q A_r(v,u,Q)(A_\infty (v,Q))^{\frac 1{p'}} \phi(A_\infty (v,Q));  \\
[u, v]_{(A_{r'}) (\psi(A_\infty)) } &:= \sup_Q A_{r'}(u,v,Q)(A_\infty (u,Q))^{\frac 1p} \psi(A_\infty (u,Q)).
\end{align*}
We have the following one supremum estimate
\begin{Theorem}\label{thm:onesupremum}
Let $(w,\sigma)$ be a pair of weights, $\lambda_Q=|Q|$, $p=q$ and $\mathcal N$ be the best constant such that \eqref{eq:norm} holds. Let $\phi,\psi$ be  increasing functions such that
\[
\int_{1/2}^\infty\Big(\frac {1}{ \phi(t) } + \frac {1}{ \psi(t)}\Big)\frac{dt}t<\infty.
\]
Then
\[
\mathcal N\lesssim [v,u]_{(A_r) (\phi(A_\infty)) }+ [u,v]_{(A_{r'}) (\psi(A_\infty)) },
\]
\end{Theorem}
For the case of $p_0=1$ and $q_0=\infty$, this was shown by Lacey and Spencer in \cite{LS}, which was referred to as the separated entropy bounds.
Now consider another definition of $A_\infty$ weights, namely,
\[
[w]_{A_\infty}^{\textup{exp}}:= \sup_{Q} \langle w\rangle_Q \exp ( \langle\log w^{-1}\rangle_Q  ).
\]
It is showed in \cite{HP} that
\[
[w]_{A_\infty}\le c_n [w]_{A_\infty}^{\textup{exp}}.
\]
We shall show that, if we replace $A_{\infty}$ with $A_{\infty}^{\textup{exp}}$, then we can relax the decay condition of $\phi$ and $\psi$ slightly. To be precise, define
 \[
 A_\infty^{\textup{exp}} (v,Q):= \langle v\rangle_Q \exp ( \langle\log v^{-1}\rangle_Q  )
 \]
 and
\begin{align*}
[v,u]_{(A_r) (\Phi(A_\infty^{\textup{exp}})) }&:= \sup_Q A_r(v,u,Q) (A_\infty^{\textup{exp}}(v,Q))^{\frac 1{p'}} \Phi(A_\infty^{\textup{exp}} (v,Q));  \\
[u, v]_{(A_{r'}) (\Psi(A_\infty^{\textup{exp}})) } &:= \sup_Q A_{r'}(u,v,Q) (A_\infty^{\textup{exp}}(u,Q))^{\frac 1{p}} \Psi(A_\infty^{\textup{exp}} (u,Q)).
\end{align*}
We have the following result
\begin{Theorem}\label{thm:exp}
Let $(w,\sigma)$ be a pair of weights, $\lambda_Q=|Q|$, $p=q$ and $\mathcal N$ be the best constant such that \eqref{eq:norm} holds. Let $\Phi,\Psi$ be  increasing functions such that
\[
\int_{1/2}^\infty\Big(\frac {1}{\Phi(t)^{ p }} + \frac {1}{\Psi(t)^{ p' }}\Big)\frac{dt}t<\infty.
\]
Then
\[
\mathcal N\lesssim [v,u]_{(A_r) (\Phi(A_\infty^{\textup{exp}})) }+ [u,v]_{(A_{r'}) (\psi(A_\infty^{\textup{exp}})) },
\]
\end{Theorem}

The organization of the paper will be as follows. In Section \ref{sec:2}, we present the proof of Theorem \ref{thm:characterization}. In section \ref{sec:3}, we prove Theorem \ref{thm:sharpconstant}. We study the one supremum estimate in Section \ref{sec:4}, which contains the proof of Theorems \ref{thm:onesupremum} and \ref{thm:exp}. And we end this manuscript with a generalized question in  Section \ref{sec:5}.

\section{A characterization of the two weight norm inequalities for bilinear forms}\label{sec:2}
In this section, we give a proof for Theorem~\ref{thm:characterization}.
Recall that  $u= w^{\frac {p_0}{p_0-p}}$ and $v=\sigma^{\frac{q_0'}{q_0'-q'}}$. We can write the two weight inequality as follows
\[
B(f, g)= \sum_{Q\in\mathcal S}\langle |f|^{p_0}\rangle_Q^{\frac 1{p_0}} \langle|g|^{q_0'}\rangle_Q^{\frac 1{q_0'}}\lambda_Q\le C \| |f|^{p_0} u^{-1}\|_{L^{p/{p_0}}(u)}^{1/{p_0}}\| |g|^{q_0'} v^{-1}\|_{L^{q'/{q_0'}}(v)}^{1/{q_0'}},
\]
which is equivalent to
\begin{equation}\label{eq:before}
\sum_{Q\in\mathcal S}(\langle |f|^{p_0}\rangle_Q^u)^{\frac 1{p_0}}  (\langle|g|^{q_0'}\rangle_Q^v)^{\frac 1{q_0'}}\lambda_Q  \langle u\rangle_Q^{\frac 1{p_0}} \langle v\rangle_Q^{\frac 1{q_0'}}  \le C \| |f|^{p_0}  \|_{L^{p/{p_0}}(u)}^{1/{p_0}}\| |g|^{q_0'} \|_{L^{q'/{q_0'}}(v)}^{1/{q_0'}},
\end{equation}
where for any function $h$ and weight $w$,
\[
\langle h\rangle_Q^w:= \frac 1{w(Q)}\int_Q h \dw.
\]
Follow the spirit in \cite{HL}, we claim that \eqref{eq:before} is equivalent to
\begin{equation}\label{eq:after}
\sum_{Q\in \mathcal S} \langle f\rangle_Q^u \langle g\rangle_Q^v \lambda_Q  \langle u\rangle_Q^{\frac 1{p_0}} \langle v\rangle_Q^{\frac 1{q_0'}}  \le C \| f \|_{L^{p }(u)} \| g \|_{L^{q' }(v)} .
\end{equation}
In fact, if \eqref{eq:before} holds, then \eqref{eq:after} follows immediately from H\"older's inequality. On the other hand, if \eqref{eq:after} holds, then we have
\begin{eqnarray*}
&&\sum_{Q\in\mathcal S}(\langle |f|^{p_0}\rangle_Q^u)^{\frac 1{p_0}}  (\langle|g|^{q_0'}\rangle_Q^v)^{\frac 1{q_0'}}\lambda_Q  \langle u\rangle_Q^{\frac 1{p_0}} \langle v\rangle_Q^{\frac 1{q_0'}}\\
&\le& \sum_{Q\in\mathcal S} \langle M_{p_0, u}^{\mathcal S}(f)\rangle_Q^u   \langle M_{q_0', v}^{\mathcal S}(g)\rangle_Q^v\lambda_Q  \langle u\rangle_Q^{\frac 1{p_0}} \langle v\rangle_Q^{\frac 1{q_0'}}\\
&\le& C \| M_{p_0, u}^{\mathcal S}(f)   \|_{L^{p }(u)} \| M_{q_0', v}^{\mathcal S}(g)  \|_{L^{q' }(v)} \\
&\le& C_{p_0,p,q_0,n} C  \| f \|_{L^{p }(u)} \| g \|_{L^{q' }(v)},
\end{eqnarray*}
where
\[
M_{p,u}^{\mathcal S}(f):=\sup_{Q\in\mathcal S} (\langle |f|^{p}\rangle_Q^u)^{\frac 1{p}}.
\]
Then by duality, it suffices to give a characterization for the following two weight norm inequality
\begin{equation}\label{eq:final}
\| T_\tau(fu)\|_{L^{q}(v)}\le C\|f\|_{L^p(u)},
\end{equation}
where
\[
 T_\tau(f )= \sum_{Q\in\mathcal S} \tau_Q \langle f \rangle_Q \mathbf 1_Q
\] and recall that
\[
\tau_Q= \langle u\rangle_Q^{-\frac 1{p_0'}} \langle v\rangle_Q^{-\frac 1{q_0}} \frac{\lambda_Q}{|Q|}.
\]
The case of $p\le q$ was due to  Lacey, Sawyer and Uriarte-Tuero \cite{LSU}, and the case of $p>q$ was given by Tanaka \cite{tanaka2014}. Here we follow the unified testing for both cases given by H\"anninen, Hyt\"onen and the author \cite{HHL}.
\begin{Proposition}\cite[Theorem 1.5]{HHL}\label{prop:char}
Let $p,q\in(1,\infty)$ and $w,\sigma$ be two measures. Let
\begin{equation}\label{eq:TandTQ}
\begin{split}
  T(f\sigma) &:=\sum_{Q\in\mathcal{D}}\lambda_Q\int_Q f  d\sigma\cdot 1_Q,\qquad\lambda_Q\geq 0,\\
  T_{Q}(f\sigma) &:=\sum_{\substack{Q'\in\mathcal{D}\\ Q'\subseteq Q}}\lambda_{Q'}\int_{Q'} f  d\sigma\cdot 1_{Q'}.
\end{split}
\end{equation}
We have
\begin{equation*}
  \|T  (\cdot \sigma)\|_{L^p(\sigma)\rightarrow L^q(w)}
  \eqsim \mathfrak{T}_s+\mathfrak{T}_s^*,
\end{equation*}
where $s\in(1,\infty]$ is determined by
\begin{equation}\label{eq:rDef}
  \frac{1}{s}=\Big(\frac{1}{q}-\frac{1}{p}\Big)_+
  :=\max\Big\{\frac{1}{q}-\frac{1}{p},0\Big\}.
\end{equation}
and
\begin{equation}\label{eq:newTesting}
\begin{split}
  \mathfrak{T}_s & :=\sup_{\mathcal F}\Big \|\Big\{\frac{\|T_{F}(\sigma)\|_{L^{q}(w)}}{\sigma(F)^{1/p}}\Big\}_{F\in\mathcal F} \Big\|_{\ell^s}, \\
  \mathfrak{T}_s^* & :=\sup_{\mathcal{G}}\Big \|\Big\{\frac{\|T_{G}(w)\|_{L^{p'}(\sigma)}}{w(G)^{1/{q'}}}\Big\}_{G\in\mathcal G} \Big\|_{\ell^s},
\end{split}
\end{equation}
where the supremums are taken over all subcollections $\mathcal F$ and $\mathcal{G}$ of $\mathcal{S}$ that are sparse with respect to $\sigma$ and $\omega$, respectively.
\end{Proposition}
Now combining \eqref{eq:final} and Proposition \ref{prop:char}, Theorem \ref{thm:characterization} follows immediately.

\section{Mixed $A_p$-$A_\infty$ type estimate for bilinear forms}\label{sec:3}
In this section, we focus on the sharp constant for the case $p=q$. In this case, the testing condition degenerate to the Sawyer type testing condition. Namely,
\begin{equation}\label{eq:sawyer}
\mathcal N\eqsim  \sup_{R\in\mathcal S}  \frac{ \| T_R (v)    \|_{L^{p'}(u)}}{v(R)^{  1/{p'} }}
 +\sup_{R\in \mathcal S}  \frac{\| T_R(u)\|_{L^p(v)}}{u(R)^{ 1/p}}.
\end{equation}
Before further estimates, we introduce the following proposition.
\begin{Proposition}\label{kolmogorov}
Let $\mathcal S$ be a sparse family and $0\le \gamma, \eta<1$ satisfying $\gamma+\eta<1$. Then
\[
\sum_{\substack{Q\in \mathcal S\\Q\subset R}}\langle u \rangle_Q^\gamma \langle v\rangle_Q^\eta |Q| \lesssim \langle u \rangle_R^\gamma \langle v\rangle_R^\eta |R|.
\]
\end{Proposition}

\begin{proof}
Indeed, set $1/r:=\gamma+\eta$, $1/s:=\gamma+(1-1/r)/2$ and $1/{s'}:=1-1/{s}$. Denote
\[
E_Q:=Q\setminus \bigcup_{Q'\in\mathcal S, Q'\subsetneq Q} Q'.
\] By sparseness and Kolmogorov's inequality, we have
\begin{align*}
\sum_{\substack{Q\in \mathcal S\\Q\subset R}}\langle u \rangle_Q^\gamma \langle v\rangle_Q^\eta |Q|&\le 2 \sum_{\substack{Q\in \mathcal S\\Q\subset R}}\langle u \rangle_Q^\gamma \langle v\rangle_Q^\eta |E_Q|\\
&\le 2 \int_{R} M(u\mathbf 1_R)^{\gamma} M(v\mathbf 1_R)^{\eta}dx\\
&\le 2 \Big(\int_R M(u\mathbf 1_R)^{s\gamma}  \Big)^{1/s}\Big(\int_R M(v\mathbf 1_R)^{s'\eta}  \Big)^{1/{s'}}\\
&\lesssim \langle u\rangle_R^\gamma |R|^{1/s} \langle v\rangle_R^{\eta}|R|^{1/{s'}}\\
&=\langle u \rangle_R^\gamma \langle v\rangle_R^\eta |R|.
\end{align*}
\end{proof}
We also need the following result
\begin{Proposition}[\cite{cov2004}, Proposition~2.2]\label{dyadicsum}
Let $1<s<\infty$, $\sigma$ be a positive Borel measure and
\[
\phi=\sum_{Q\in\mathcal D} \alpha_Q \mathbf 1_Q,\qquad \phi_Q=\sum_{Q'\subset Q}\alpha_{Q'} \mathbf 1_{Q'}.
\]
Then
\[
\|\phi\|_{L^s(\sigma)}\eqsim \Big( \sum_{Q\in \mathcal D} \alpha_Q (\langle\phi_Q\rangle_Q^\sigma)^{s-1}\sigma(Q) \Big)^{1/s}.
\]
\end{Proposition}
Now we are ready to prove Theorem \ref{thm:sharpconstant}.
\begin{proof}[Proof of Theorem \ref{thm:sharpconstant}]
As that in \cite{BFP}, we denote by $\rho$ the critical index $1+ p_0/{q_0'}$. First, we consider the case $p\ge \rho$. set
\[
\alpha= \frac{1}{p_0(r-1)}\min\{p-p_0, 1\}.
\]
We can check that
\[
\frac 1{p_0}- (r-1)\alpha\ge 0,\,\, \frac 1{q_0'}-\alpha\ge 0,\,\,\mbox{and}\,\, \frac 1{p_0}- (r-1)\alpha+\frac 1{q_0'}-\alpha<1.
\]
By Propositions \ref{dyadicsum} and  \ref{kolmogorov}, we have
\begin{eqnarray*}
 &&\| T_R(v)  \|_{L^{p'}(u)}\\
&=&\Big \|\sum_{\substack{Q\in\mathcal S\\ Q\subset R}}  \langle u\rangle_Q^{-\frac 1{p_0'}} \langle v\rangle_Q^{\frac 1{q_0'}} \mathbf 1_Q   \Big\|_{L^{p'}(u)} \\
&\eqsim &\Big(    \sum_{\substack{Q\in\mathcal S\\ Q\subset R}}  \langle u\rangle_Q^{-\frac 1{p_0'}} \langle v\rangle_Q^{\frac 1{q_0'}} u(Q)  \Big( \frac 1{u(Q)}\sum_{Q'\subset Q} \langle u\rangle_{Q'}^{-\frac 1{p_0'}} \langle v\rangle_{Q'}^{\frac 1{q_0'}} u(Q')\Big)^{p'-1}      \Big)^{\frac 1{p'}}\\
&\le& [v, u]_{A_r}^{\frac{\alpha}{p} }\Big(    \sum_{\substack{Q\in\mathcal S\\ Q\subset R}}  \langle u\rangle_Q^{-\frac 1{p_0'}} \langle v\rangle_Q^{\frac 1{q_0'}} u(Q)  \Big( \frac 1{u(Q)}\sum_{Q'\subset Q} \langle u\rangle_{Q'}^{\frac 1{p_0}-(r-1)\alpha} \langle v\rangle_{Q'}^{\frac 1{q_0'}-\alpha } |Q'|\Big)^{p'-1}      \Big)^{\frac 1{p'}}\\
&\lesssim& [v, u]_{A_r}^{\frac{\alpha}{p } }\Big(    \sum_{\substack{Q\in\mathcal S\\ Q\subset R}}  \langle u\rangle_Q^{-\frac 1{p_0'}} \langle v\rangle_Q^{\frac 1{q_0'}} u(Q)  \Big( \frac 1{u(Q)} \langle u\rangle_Q^{\frac 1{p_0}-(r-1)\alpha} \langle v\rangle_{Q }^{\frac 1{q_0'}-\alpha }  |Q |\Big)^{p'-1}      \Big)^{\frac 1{p'}}\\
&=& [v, u]_{A_r}^{\frac{\alpha}{p } }\Big(    \sum_{\substack{Q\in\mathcal S\\ Q\subset R}}  \langle u\rangle_Q^{\frac 1{p_0}-(p'-1)((r-1)\alpha+\frac 1{p_0'})} \langle v\rangle_Q^{\frac 1{q_0'}+(\frac 1{q_0'}-\alpha)(p'-1)} |Q|  \Big)^{\frac 1{p'}}\\
&\le& [v, u]_{A_r}^{\frac{\alpha}{p }+\frac 1{p'p_0(r-1)}- (\frac\alpha p+\frac 1{pp_0'(r-1)}) }\Big(\sum_{\substack{Q\in\mathcal S\\ Q\subset R}}v(Q)\Big)^{\frac 1{p'}} \\
&\lesssim& [v, u]_{A_r}^{\frac 1{q_0'}-\frac 1{p'}}[v]_{A_\infty}^{\frac 1{p'}} v(R)^{\frac 1{p'}},
\end{eqnarray*}
For the case $p<\rho$,  set
\[
\tilde \alpha =(\frac 1{q_0'}-\frac 1{p'})p.
\]
Again, we can check that
\[
\frac 1{q_0'}-\tilde \alpha=\frac {1}{q_0'}(p'-1)(p-1)-\frac {1}{q_0'}(p'-q_0')(p-1)\ge0,
\]
\[
\frac 1{p_0}-\tilde\alpha(r-1)\ge 0,
\]
and
\[
\frac 1{p_0}-\tilde\alpha(r-1)+ \frac 1{q_0'}-\tilde \alpha= 1-(p-1)(\frac 1{p_0}-\frac 1{q_0})<1.
\]
By Propositions \ref{dyadicsum} and  \ref{kolmogorov} again, we have
\begin{eqnarray*}
 &&\| T_R(v)  \|_{L^{p'}(u)}\\
&\eqsim& \Big(    \sum_{\substack{Q\in\mathcal S\\ Q\subset R}}  \langle u\rangle_Q^{-\frac 1{p_0'}} \langle v\rangle_Q^{\frac 1{q_0'}} u(Q)  \Big( \frac 1{u(Q)}\sum_{Q'\subset Q} \langle u\rangle_{Q'}^{\frac 1{p_0}} \langle v\rangle_{Q'}^{\frac 1{q_0'}} |Q'|\Big)^{p'-1}      \Big)^{\frac 1{p'}}\\
&\le& [v,u]_{A_r}^{\frac {\tilde \alpha}p}\Big(    \sum_{\substack{Q\in\mathcal S\\ Q\subset R}}  \langle u\rangle_Q^{-\frac 1{p_0'}} \langle v\rangle_Q^{\frac 1{q_0'}} u(Q)  \Big( \frac 1{u(Q)}\sum_{Q'\subset Q} \langle u\rangle_{Q'}^{\frac 1{p_0}-\tilde\alpha(r-1)} \langle v\rangle_{Q'}^{\frac 1{q_0'}-\tilde \alpha } |Q'|\Big)^{p'-1}      \Big)^{\frac 1{p'}}\\
&\lesssim& [v,u]_{A_r}^{\frac {\tilde \alpha}p}\Big(    \sum_{\substack{Q\in\mathcal S\\ Q\subset R}}  \langle u\rangle_Q^{-\frac 1{p_0'}} \langle v\rangle_Q^{\frac 1{q_0'}} u(Q)  \Big( \frac 1{u(Q)}  \langle u\rangle_{Q }^{\frac 1{p_0}-\tilde\alpha(r-1)} \langle v\rangle_{Q }^{\frac 1{q_0'}-\tilde \alpha } |Q |\Big)^{p'-1}      \Big)^{\frac 1{p'}}\\
&= &[v,u]_{A_r}^{\frac {\tilde \alpha}p}\Big(    \sum_{\substack{Q\in\mathcal S\\ Q\subset R}}  \langle u\rangle_Q^{\frac 1{p_0}- (p'-1)(\frac 1{p_0'}+\tilde \alpha(r-1))} \langle v\rangle_Q^{\frac 1{q_0'}+(p'-1)(\frac 1{q_0'}-\tilde\alpha)} |Q|      \Big)^{\frac 1{p'}}\\
&=&[v,u]_{A_r}^{\frac 1{q_0'}-\frac 1{p'}} (\sum_{\substack{Q\in\mathcal S\\ Q\subset R}}v(Q))^{\frac 1{p'}}\\
&\lesssim& [v, u]_{A_r}^{\frac 1{q_0'}-\frac 1{p'}}[v]_{A_\infty}^{\frac 1{p'}} v(R)^{\frac 1{p'}}.
\end{eqnarray*}
By symmetry, we have
\begin{align*}
\| T_R(u)  \|_{L^{p }(v)}
&\lesssim [u, v]_{A_{r'}}^{\frac 1{p_0}-\frac 1{p}}[u]_{A_\infty}^{\frac 1{p}} u(R)^{\frac 1{p}}\\
&= [v, u]_{A_r}^{\frac 1{q_0'}-\frac 1{p'}}[u]_{A_\infty}^{\frac 1{p}} u(R)^{\frac 1{p}}.
\end{align*}
Then the desired estimate follows from \eqref{eq:sawyer} immediately.
\end{proof}

\section{Mixed $A_p$-$A_\infty$ type estimates with one supremum}\label{sec:4}
In this section, we study the one supremum estimate. And this could be done by just slightly modify the arguments in the previous section. We first prove Theorem~\ref{thm:onesupremum}.
\begin{proof}[Proof of Theorem~\ref{thm:onesupremum}]
 By symmetry, we only need to estimate
$\| T_R(v)  \|_{L^{p'}(u)}$. Denote
\[
\mathcal S_a=\{Q\in \mathcal S: 2^a\le A_\infty(v,Q)< 2^{a+1}\}.
\]
We will abuse of using the notations $T$ and $ T_R(v)$ slightly, which is now understood as summation over $\mathcal S_a$ instead of $\mathcal S$.
We still consider the case $p\ge \rho$ first, we have
\begin{eqnarray*}
 &&\| T_R(v)  \|_{L^{p'}(u)}\\
&\eqsim& \Big(    \sum_{\substack{Q\in\mathcal S_a\\ Q\subset R}}  \langle u\rangle_Q^{-\frac 1{p_0'}} \langle v\rangle_Q^{\frac 1{q_0'}} u(Q)  \Big( \frac 1{u(Q)}\sum_{Q'\subset Q} \langle u\rangle_{Q'}^{\frac 1{p_0}} \langle v\rangle_{Q'}^{\frac 1{q_0'}} |Q'|\Big)^{p'-1}      \Big)^{\frac 1{p'}}\\
&\le& \bigg(\frac{[v,u]_{(A_r) (\phi(A_\infty)) }}{2^{\frac a{p'}}\phi(2^a)}\bigg)^{\frac{\alpha   q_0}{q_0-p}}\Big(    \sum_{\substack{Q\in\mathcal S_a\\ Q\subset R}}  \langle u\rangle_Q^{-\frac 1{p_0'}} \langle v\rangle_Q^{\frac 1{q_0'}} u(Q)  \Big( \frac 1{u(Q)}\\
&&\qquad\times\sum_{Q'\subset Q} \langle u\rangle_{Q'}^{\frac 1{p_0}-(r-1)\alpha} \langle v\rangle_{Q'}^{\frac 1{q_0'}-\alpha } |Q'|\Big)^{p'-1}      \Big)^{\frac 1{p'}}\\
&\lesssim& \bigg(\frac{[v,u]_{(A_r) (\phi(A_\infty)) }}{2^{\frac a{p'}}\phi(2^a)}\bigg)^{\frac{\alpha   q_0}{q_0-p}}\Big(    \sum_{\substack{Q\in\mathcal S_a\\ Q\subset R}}  \langle u\rangle_Q^{-\frac 1{p_0'}} \langle v\rangle_Q^{\frac 1{q_0'}} u(Q) \\
 &&\qquad\times\Big( \frac 1{u(Q)} \langle u\rangle_Q^{\frac 1{p_0}-(r-1)\alpha} \langle v\rangle_{Q }^{\frac 1{q_0'}-\alpha }  |Q |\Big)^{p'-1}      \Big)^{\frac 1{p'}}\\
&=& \bigg(\frac{[v,u]_{(A_r) (\phi(A_\infty)) }}{2^{\frac a{p'}}\phi(2^a)}\bigg)^{\frac{\alpha   q_0}{q_0-p}}\Big(    \sum_{\substack{Q\in\mathcal S_a\\ Q\subset R}}  \langle u\rangle_Q^{\frac 1{p_0}-(p'-1)((r-1)\alpha+\frac 1{p_0'})} \langle v\rangle_Q^{\frac 1{q_0'}+(\frac 1{q_0'}-\alpha)(p'-1)} |Q|  \Big)^{\frac 1{p'}}\\
&\le &  \frac{[v,u]_{(A_r) (\phi(A_\infty)) }}{2^{\frac a{p'}}\phi(2^a)}    \Big(\sum_{\substack{Q\in\mathcal S_a\\ Q\subset R}}v(Q)\Big)^{\frac 1{p'}} \\
&\lesssim & \frac{1}{\phi(2^a) } [v,u]_{(A_r) (\phi(A_\infty)) } v(R)^{\frac 1{p'}}.
\end{eqnarray*}

For the case $p<\rho$, similar arguments show that
\[
\| T_R(v)  \|_{L^{p'}(u)}\lesssim \frac{1}{\phi(2^a) } [v,u]_{(A_r) (\phi(A_\infty)) } v(R)^{\frac 1{p'}}.
\]
It remains to sum over $a$, by definition, $a \ge 0$, we have
\begin{align*}
 \sup_{R\in\mathcal S}  \frac{ \| T_R (v)    \|_{L^{p'}(u)}}{v(R)^{  1/{p'} }}&\lesssim \sum_{a\ge 0 }\frac{1}{\phi(2^a) } [v,u]_{(A_r) (\phi(A_\infty)) }\\
 &\lesssim \sum_{a\ge 0}\int_{2^{a-1}}^{2^a}\frac {1}{\phi(t)  }\frac{dt}{t}[v,u]_{(A_r) (\phi(A_\infty)) }\\
 &= \int_{1/2}^{\infty} \frac {1}{\phi(t)  }\frac{dt}{t}[v,u]_{(A_r) (\phi(A_\infty)) }.
\end{align*}
This completes the proof.
\end{proof}

Next we prove Theorem \ref{thm:exp}.
\begin{proof}[Proof of Theorem \ref{thm:exp}]
We follow the same strategy as that in the proof of Theorem \ref{thm:onesupremum}. Set
\[
\mathcal S_a=\{Q\in \mathcal S: 2^a\le A_\infty^{\textup{exp}}(v,Q)< 2^{a+1}\}.
\]
 We have
 \begin{eqnarray*}
 \| T_R(v)  \|_{L^{p'}(u)}
&\lesssim& \sum_{a\ge 0}  \frac{[v,u]_{(A_r) (\Phi(A_\infty)) }}{2^{\frac a{p'}}\Phi(2^a)}\Big(\sum_{\substack{Q\in\mathcal S_a\\ Q\subset R}}v(Q)\Big)^{\frac 1{p'}}\\
&\le& [v,u]_{(A_r) (\Phi(A_\infty)) }\sum_{a\ge 0}  \frac{1}{\Phi(2^a)}\Big(\sum_{\substack{Q\in\mathcal S_a\\ Q\subset R}}\exp\Big( \dashint_Q \log v \Big) |Q|\Big)^{\frac 1{p'}}\\
&\le & [v,u]_{(A_r) (\Phi(A_\infty)) }\Big(\sum_{a\ge 0}  \frac{1}{\Phi(2^a)^p}\Big)^{\frac 1p} \Big(\sum_{\substack{Q\in\mathcal S \\ Q\subset R}}\exp\Big( \dashint_Q \log v \Big) |Q|\Big)^{\frac 1{p'}}\\
&\lesssim& [v,u]_{(A_r) (\Phi(A_\infty)) } v(R)^{\frac 1{p'}}\Big(\int_{1/2}^\infty \frac {1}{\Phi(t)^{ p}}\frac {dt}t\Big)^{\frac 1p}.
\end{eqnarray*}
This completes the proof.
\end{proof}

\section{Further discussions}\label{sec:5}
In this section, we propose a new conjecture which implies both the one supremum conjecture (with $A_\infty$ replaced by $A_\infty^{\textup{exp}}$) and the separated bump conjecture. On the other hand, it shares many similar results as the later two problems. To be precise, define
\begin{align*}
[u,v]_{A, q_0,p,r}&= \sup_{Q}\langle v\rangle_Q^{ \frac 1p-\frac 1{q_0} } \langle u\rangle_Q^{(\frac 1p-\frac 1{q_0})(r-1)} \frac {\langle u\rangle_Q^{\frac 1{p}}}{\langle u^{\frac 1p}\rangle_{A, Q}},\\
[v,u]_{B,  p_0',p',r'}&= \sup_{Q}\langle u\rangle_Q^{\frac 1{p'}-\frac 1{p_0'}} \langle v\rangle_Q^{(\frac 1{p'}-\frac 1{p_0'})(r'-1)} \frac {\langle v\rangle_Q^{\frac 1{p'}}}{\langle v^{\frac 1{p'}}\rangle_{B, Q}},
\end{align*}
where $A\in B_p$ and $B\in B_{p'}$. Recall that we say a Young function $A$ belongs to $B_p$ if
\[
\int_{1/2}^\infty \frac{A(t)}{t^p}\frac {dt}{t}<\infty,
\]
and the Luxembourg norm $\langle f\rangle_{A,Q}$ is defined by
\[
\langle f\rangle_{A,Q}:=\inf \{\lambda>0: \langle A(f/\lambda)\rangle_{Q}\le 1 \}.
\]
We propose the following conjecture.
\begin{Conjecture}\label{conjecture-l}
Let $(w,\sigma)$ be a pair of weights, $\lambda_Q=|Q|$, $p=q$ and $\mathcal N$ be the best constant such that \eqref{eq:norm} holds. Let $A\in B_p$ and $B\in B_{p'}$. Then there holds
\[
\mathcal N \le C ([u,v]_{A, q_0,p,r}+[v,u]_{B, p_0', p',r'}),
\]
where the constant $C>0$ is independent of $w$ and $\sigma$.
\end{Conjecture}
Now we shall see that this conjecture implies both the one supremum conjecture and also the separated bump conjecture. For general indices $p_0$ and $q_0$,
the separated bump conjecture can be stated as follows
\begin{Conjecture}
Let $(w,\sigma)$ be a pair of weights, $\lambda_Q=|Q|$, $p=q$ and $\mathcal N$ be the best constant such that \eqref{eq:norm} holds. Let $A\in B_p$ and $B\in B_{p'}$. Then there holds
\[
\mathcal N \le C ((u,v)_{A, q_0,p,r}+(v,u)_{B, p_0', p',r'}),
\]
where the constant $C>0$ is independent of $w$ and $\sigma$ and
\begin{align*}
(u,v)_{A, q_0,p,r}&= \sup_{Q}\langle v\rangle_Q^{ \frac 1p-\frac 1{q_0} } \langle u\rangle_Q^{(\frac 1p-\frac 1{q_0})(r-1)} \frac {\langle u^{\frac 1{p'}}\rangle_{\bar A, Q}}{\langle u\rangle_Q^{\frac 1{p'}}},\\
(v,u)_{B,  p_0',p',r'}&= \sup_{Q}\langle u\rangle_Q^{\frac 1{p'}-\frac 1{p_0'}} \langle v\rangle_Q^{(\frac 1{p'}-\frac 1{p_0'})(r'-1)} \frac {\langle v^{\frac 1{p}}\rangle_{\bar B, Q}}{\langle v\rangle_Q^{\frac 1{p}}}.
\end{align*}
\end{Conjecture}
Indeed, by general H\"older's inequality,
\[
\frac {\langle u\rangle_Q^{\frac 1{p}}}{\langle u^{\frac 1p}\rangle_{A, Q}}\le \frac { \langle u^{\frac 1{p'}}\rangle_{\bar A, Q}}{\langle u \rangle_{ Q}^{\frac 1{p'}}},\,\,
\frac {\langle v\rangle_Q^{\frac 1{p'}}}{\langle v^{\frac 1{p'}}\rangle_{B, Q}}\le \frac { \langle v^{\frac 1{p}}\rangle_{\bar B, Q}}{\langle v \rangle_{ Q}^{\frac 1{p}}},
\]
which means Conjecture \ref{conjecture-l} implies separated bump conjecture. On the other hand, by Jensen's inequality,
\[
\langle u^{\frac 1p}\rangle_{A, Q}\ge \exp\Big(\dashint_Q \log u^{\frac 1p}\Big)=\Big( \exp\Big(\dashint_Q \log u \Big)\Big)^{\frac 1p}.
\]
Then
\[
\frac {\langle u\rangle_Q^{\frac 1{p}}}{\langle u^{\frac 1p}\rangle_{A, Q}}\le A_\infty^{\textup{exp}}(u,Q)^{\frac 1p},
\]
which means Conjecture \ref{conjecture-l} implies also the one supremum conjecture.

Although Conjecture \ref{conjecture-l} is stronger than both separated bump conjecture and one supremum conjecture, it still contain the essential property as the separated bump conjecture and one supremum conjecture. Indeed, we are free to write
\[
\langle \sigma^{\frac 1p}\rangle_{A, Q}=\langle \sigma\rangle_{\tilde A, Q}^{\frac 1p},
\]
where $\tilde A(t)= A(t^{1/p})$. Since $M_{A}^{\mathcal D}$ is bounded on $L^p$ (see \cite{P}), we have $M_{\tilde A}^{\mathcal D}$ is bounded on $L^1$, and this is the key point.
We shall see this by showing the following result.

\begin{Theorem}
Let $(w,\sigma)$ be a pair of weights, $\lambda_Q=|Q|$, $p=q$ and $\mathcal N$ be the best constant such that \eqref{eq:norm} holds. Let $A\in B_p$, $B\in B_{p'}$
 and  $\phi,\psi$ be  increasing functions such that
\[
\int_{1/2}^\infty \left(\frac 1{\phi(t)^p } +  \frac 1{\psi(t)^{p'}}\right) \frac{dt}t<\infty.
\]
Then
\[
\mathcal N\lesssim [u,v]_{A, q_0,p,r,\psi}+ [v,u]_{B, p_0',p',r',\phi},
\]
where
\begin{align*}
[u,v]_{A,q_0,p,r,\psi}&= \sup_{Q}\langle v\rangle_Q^{ \frac 1p-\frac 1{q_0} } \langle u\rangle_Q^{(\frac 1p-\frac 1{q_0})(r-1)} \frac {\langle u\rangle_Q^{\frac 1{p}}}{\langle u^{\frac 1p}\rangle_{A, Q}}\psi\Big( \frac {\langle u\rangle_Q^{\frac 1{p}}}{\langle u^{\frac 1p}\rangle_{A, Q}} \Big),\\
[v,u]_{B, p_0',p',r',\phi}&= \sup_{Q}\langle u\rangle_Q^{\frac 1{p'}-\frac 1{p_0'}} \langle v\rangle_Q^{(\frac 1{p'}-\frac 1{p_0'})(r'-1)} \frac {\langle v\rangle_Q^{\frac 1{p'}}}{\langle v^{\frac 1{p'}}\rangle_{B, Q}}\phi\Big( \frac {\langle v\rangle_Q^{\frac 1{p'}}}{\langle v^{\frac 1{p'}}\rangle_{B, Q}}\Big).
\end{align*}
\end{Theorem}
\begin{proof}
The proof is quite similar to the proof of Theorem \ref{thm:exp}. Set
\[
\mathcal S_a:=\{Q\in \mathcal S: 2^a\le \frac {\langle v\rangle_Q^{\frac 1{p'}}}{\langle v^{\frac 1{p'}}\rangle_{B, Q}}< 2^{a+1}\}.
\]
Then we have
\begin{eqnarray*}
 \| T_R(v)  \|_{L^{p'}(u)}
&\lesssim& \sum_{a\ge 0}  \frac{[v,u]_{B, p_0',p',r',\phi}}{2^{ a }\phi(2^a)}\Big(\sum_{\substack{Q\in\mathcal S_a\\ Q\subset R}}v(Q)\Big)^{\frac 1{p'}}\\
&\le& [v,u]_{B, p_0',p',r',\phi}\sum_{a\ge 0}  \frac{1}{\phi(2^a)}\Big(\sum_{\substack{Q\in\mathcal S_a\\ Q\subset R}}\langle v^{\frac 1{p'}}\rangle_{B, Q}^{p'} |Q|\Big)^{\frac 1{p'}}\\
&\le & [v,u]_{B, p_0',p',r',\phi}\Big(\sum_{a\ge 0}  \frac{1}{\phi(2^a)^p}\Big)^{\frac 1p} \Big(\sum_{\substack{Q\in\mathcal S \\ Q\subset R}}\langle v^{\frac 1{p'}}\rangle_{B, Q}^{p'} |Q|\Big)^{\frac 1{p'}}\\
&\lesssim& [v,u]_{B, p_0',p',r',\phi} v(R)^{\frac 1{p'}} \|M_B\|_{L^{p'}\rightarrow L^{p'}} \Big(\int_{1/2}^\infty \frac {1}{\phi(t)^{ p}}\frac {dt}t\Big)^{\frac 1p}.
\end{eqnarray*}
Then the desired result follows immediately from Theorem \ref{thm:characterization}.
\end{proof}
\begin{Remark}
In \cite{Lacey}, Lacey showed the same result with $\langle u\rangle_Q^{\frac 1{p}}/\langle u^{\frac 1p}\rangle_{A, Q}$ replaced by $\langle u^{\frac 1{p'}}\rangle_{\bar A, Q}/\langle u\rangle_Q^{\frac 1{p'}}$ and analogously for $v$ when $p_0=1$ and $q_0=\infty$ (Recall that in this case, $u=\sigma$ and $v=w$). And therefore, our estimate improves Lacey's bound.
On the other hand, it is easy to see that it improves Theorem \ref{thm:exp} as well.
\end{Remark}
We also have the following result.
\begin{Theorem}
Let $(w,\sigma)$ be a pair of weights, $\lambda_Q=|Q|$, $p=q$ and $\mathcal N$ be the best constant such that \eqref{eq:norm} holds. Let $A\in B_p$, $B\in B_{p'}$.
Then
\[
\mathcal N\lesssim [u,v]_{A, B},
\]
where
\[
[u,v]_{A, B}= \sup_{Q}\langle v\rangle_Q^{ \frac 1p-\frac 1{q_0} } \langle u\rangle_Q^{(\frac 1p-\frac 1{q_0})(r-1)} \frac {\langle u\rangle_Q^{\frac 1{p}}}{\langle u^{\frac 1p}\rangle_{A, Q}}\cdot\frac {\langle v\rangle_Q^{\frac 1{p'}}}{\langle v^{\frac 1{p'}}\rangle_{B, Q}}.
\]
\end{Theorem}
\begin{proof}
Instead of using Proposition \ref{dyadicsum},  we shall use the technique of parallel stopping cubes here.   We define the principal cubes $\mathcal F$ for $(f, u)$  as follows
\begin{eqnarray*}
\mathcal F&:=& \bigcup_{k=0}^\infty \mathcal F_k, \quad \mathcal F_0:=  \{\textup{maximal cubes in }\mathcal S\}\\
\mathcal F_{k+1}&:=& \bigcup_{F\in \mathcal F_k}\ch_{\mathcal F}(F),\quad \ch_{\mathcal F}(F):= \{ Q\subsetneq F\, \textup{maximal \,s.t.} \langle f\rangle_Q^u>2\langle f\rangle_F^u \},
\end{eqnarray*}
and analogously define $\mathcal G$ for $(g, v)$. We also denote by  $\pi_{\mathcal F} (Q)$ the minimal cube in $\mathcal F$ which contains $Q$, and
$\pi (Q)=(F, G)$ if $\pi_{\mathcal F}(Q)=F$ and $\pi_{\mathcal G}(Q)=G$. By construction, we have
\begin{equation}\label{eq:stopping}
\sum_{F\in \mathcal F} (\langle f\rangle_F^u)^p u(F) \lesssim \|f\|_{L^p(u)}^p.
\end{equation} Now we start our arguments from \eqref{eq:after}. We have
\begin{align*}
 \sum_{Q\in \mathcal S} \langle f\rangle_Q^u \langle g\rangle_Q^v   \langle u\rangle_Q^{\frac 1{p_0}} \langle v\rangle_Q^{\frac 1{q_0'}}|Q|
&= \sum_{F\in \mathcal F}\sum_{G\in \mathcal G}\sum_{\pi(Q)=(F,G)} \langle f\rangle_Q^u \langle g\rangle_Q^v    \langle u\rangle_Q^{\frac 1{p_0}} \langle v\rangle_Q^{\frac 1{q_0'}}|Q|\\
&\lesssim   \sum_{F\in \mathcal F} \langle f\rangle_F^u \sum_{\substack{G\in \mathcal G\\ G\subset F}} \langle g\rangle_G^v \sum_{\pi(Q)=(F,G)}   \langle u\rangle_Q^{\frac 1{p_0}} \langle v\rangle_Q^{\frac 1{q_0'}}|Q|\\
&+ \sum_{G\in \mathcal G} \langle g\rangle_G^v \sum_{\substack{F\in \mathcal F\\ F\subset G}}  \langle f\rangle_F^u\sum_{\pi(Q)=(F,G)}   \langle u\rangle_Q^{\frac 1{p_0}} \langle v\rangle_Q^{\frac 1{q_0'}}|Q|\\
&:= I+II.
\end{align*}
By symmetry, we only need to estimate $I$. We have
\begin{eqnarray*}
&&\sum_{\pi(Q)=(F,G)} \langle u\rangle_Q^{\frac 1{p_0}} \langle v\rangle_Q^{\frac 1{q_0'}}|Q|\\
&= &\sum_{\pi(Q)=(F,G)} \langle u\rangle_Q^{\frac 1{p_0}-\frac 1p} \langle v\rangle_Q^{\frac 1{q_0'}-\frac 1{p'}}\frac {\langle u\rangle_Q^{\frac 1{p}}}{\langle u^{\frac 1p}\rangle_{A, Q}}\cdot\frac {\langle v\rangle_Q^{\frac 1{p'}}}{\langle v^{\frac 1{p'}}\rangle_{B, Q}}\langle u^{\frac 1p}\rangle_{A, Q}\cdot\langle v^{\frac 1{p'}}\rangle_{B, Q} |Q|\\
&\le &[u,v]_{A, B}\Big(\sum_{\pi(Q)=(F,G)} \langle u^{\frac 1p}\rangle_{A, Q}^p |Q|\Big)^{\frac 1p}\Big( \sum_{\pi(Q)=(F,G)}\langle v^{\frac 1{p'}}\rangle_{B, Q}^{p'}|Q| \Big)^{\frac 1{p'}}\\
&\lesssim&[u,v]_{A, B} \Big(\sum_{\pi(Q)=(F,G)} \langle u^{\frac 1p}\rangle_{A, Q}^p |Q|\Big)^{\frac 1p} v(G)^{\frac 1{p'}}.
\end{eqnarray*}
Then by H\"older's inequality,
\begin{align*}
I&\le [u,v]_{A, B} \sum_{F\in \mathcal F} \langle f\rangle_F^u \Big(\sum_{\substack{G\in \mathcal G\\ G\subset F}}\sum_{\pi(Q)=(F,G)} \langle u^{\frac 1p}\rangle_{A, Q}^p |Q|\Big)^{\frac 1p} \Big(\sum_{\substack{G\in \mathcal G\\ \pi_{\mathcal F}(G)= F}} (\langle g\rangle_G^v )^{p'} v(G)\Big)^{\frac 1{p'}}\\
&\lesssim [u,v]_{A, B} \sum_{F\in \mathcal F} \langle f\rangle_F^u u(F)^{\frac 1p}  \Big(\sum_{\substack{G\in \mathcal G\\ \pi_{\mathcal F}(G)= F}} (\langle g\rangle_G^v )^{p'} v(G)\Big)^{\frac 1{p'}}\\
&\le [u,v]_{A, B} \Big(\sum_{F\in \mathcal F} (\langle f\rangle_F^u)^p u(F)\Big)^{\frac 1p}  \Big(\sum_{F\in \mathcal F}\sum_{\substack{G\in \mathcal G\\ \pi_{\mathcal F}(G)= F}} (\langle g\rangle_G^v )^{p'} v(G)\Big)^{\frac 1{p'}}\\
&\lesssim [u,v]_{A, B} \|f\|_{L^p(u)} \|g\|_{L^{p'}(v)},
\end{align*}
where \eqref{eq:stopping} is used in the last step.
\end{proof}

\begin{Remark}
In \cite{Lerner1}, Lerner proved the so-called bump conjecture, i.e., the same result with $\langle u\rangle_Q^{\frac 1{p}}/\langle u^{\frac 1p}\rangle_{A, Q}$ replaced by $\langle u^{\frac 1{p'}}\rangle_{\bar A, Q}/\langle u\rangle_Q^{\frac 1{p'}}$ and analogously for $v$ when $p_0=1$ and $q_0=\infty$. Therefore, our result improves the bump theorem. In \cite{LM}, Lerner and Moen also proved the following estimate, for any Calder\'on-Zygmund operator $T$, there holds
\begin{align*}
\|T\|_{L^p(w)}&\le \sup_Q (\langle w\rangle_Q \langle w^{1-p'}\rangle_Q^{p-1})^{ \frac 1{p-1}} A_\infty^{\textup{exp}}(w)(Q)^{1-\frac 1{p-1}}\\&=\sup_{Q} \langle w\rangle_Q^{\frac 1p}\langle w^{1-p'}\rangle_Q^{\frac 1{p'}}A_\infty^{\textup{exp}}(w,Q)^{\frac 1{p'}} A_\infty^{\textup{exp}}(w^{1-p'},Q)^{\frac 1p}.
\end{align*}
Therefore, our result improves the above estimate as well.
\end{Remark}

\textbf{Acknowledgements}.\,\, The author would like to thank Prof. Tuomas P. Hyt\"onen for suggesting this problem and for many helpful discussions which improve the quality of this paper.

\end{document}